\documentclass[11pt]{amsart}
\usepackage{amsmath}
\usepackage{bbm}
\usepackage{mathrsfs}
\usepackage[english]{babel}
\usepackage[utf8x]{inputenc}
\usepackage[T1]{fontenc}

\usepackage[top=3cm,bottom=2cm,left=2.5cm,right=2.5cm,marginparwidth=1.75cm]{geometry}

\usepackage{amsmath}
\usepackage{graphicx}
\usepackage[colorinlistoftodos]{todonotes}
\usepackage[colorlinks=true, allcolors=blue]{hyperref}

\newtheorem{thm}{Theorem}[section]
\newtheorem{prop}[thm]{Proposition}
\newtheorem{cor}[thm]{Corollary}
\newtheorem{lemma}[thm]{Lemma}
\newtheorem{defi}{Definition}

\theoremstyle{remark} 
\newtheorem{remark}[]{Remark}

\usepackage{mathtools} 
\mathtoolsset{showonlyrefs,showmanualtags}
\numberwithin{equation}{section}
\usepackage{amsmath}

\newcommand{\e}{\varepsilon}

\newcommand{\U}{\mathcal{U}}
\newcommand{\Pc}{\mathcal{P}}
\newcommand{\G}{\mathcal{G}}
\newcommand{\M}{M}
\newcommand{\N}{\mathcal{N}}

\newcommand{\hB}{\hat{B}}
\newcommand{\hG}{\hat{\mathcal{G}}}

\newcommand{\p}{\partial}
\newcommand{\Omegabad}{\Omega_{x,R_j,n}}

\newcommand{\EE}{\mathbb{E}}
\newcommand{\R}{\mathbb{R}}

\newcommand{\NN}{\mathbb{N}}
\newcommand{\PP}{\mathbb{P}}

\newcommand{\Vol}{\mathrm{Vol}}
\newcommand{\B}{\hat{B}}

\newcommand{\be}{\begin{equation}}
\newcommand{\ee}{\end{equation}}

\title[Topologies of random geometric complexes]{Topologies of random geometric complexes on Riemannian manifolds 
in the thermodynamic limit}
\author{Antonio Auffinger, Antonio Lerario, Erik Lundberg}

\begin{document}

\begin{abstract}
We investigate the topologies
of random geometric complexes
built over
random points sampled on Riemannian manifolds
in the so-called ``thermodynamic'' regime.
We prove the existence of 
universal limit laws for the topologies;
namely, the random normalized counting measure of connected components (counted according to homotopy type) is shown to converge in probability to a deterministic probability measure.
Moreover, we show that the support of the
deterministic limiting measure 
equals the set of all homotopy types for Euclidean connected geometric complexes of the same dimension as the manifold.

\end{abstract}
\maketitle

\section{Introduction}

Sarnak and Wigman \cite{SarnakWigman} recently established, utilizing methods developed by Nazarov and Sodin \cite{NazarovSodin},
the existence of universal limit laws for the topologies of nodal sets of random band-limited functions on Riemannian manifolds. In the current paper, we adapt these methods to the setting of random geometric complexes, that is,
simplicial complexes with vertices
arising from a random point process
and faces determined by distances between vertices.

Kahle \cite{Kahle2011} made the first extensive investigation into the topology of random geometric complexes  generated by a point process in Euclidean space (zero-dimensional homology of random geometric graphs were also investigated earlier in \cite{Penrose}).
The expectation of each Betti number is studied within three main phases or regimes based on the relation between density of points and radius of the neighborhoods determining the complex: the subcritical regime (or ``dust phase'') where there are many connected components with little topology, the critical regime (or ``thermodynamic regime'') where topology is the richest (and where the percolation threshold appears), and the supercritical regime where the connectivity threshold appears.
The thermodynamic regime is seen to have the most intricate topology.
Many cycles of various dimensions begin to form as we enter this regime and many cycles become boundaries as we leave this regime.

Random geometric complexes on Riemannian manifolds were studied earlier in the influential work \cite{NSW} of Niyogi, Smale, and Weinberger,
where the manifold is embedded in Euclidean space and the distance between vertices is given by the ambient Euclidean distance.\footnote{
In the current paper, where our manifold is not necessarily embedded, we use geodesic distance to build the complexes.  If the manifold happens to be embedded, it can be seen, using Lemma \ref{lemma:convex}, that the same
limit law stated in Theorem \ref{thm:main1} holds when using the ambient Euclidean distance.}
The main question in \cite{NSW}
is motivated by applications in ``manifold learning'' and concerns the recovery of the topology of a manifold via a random sample of points on the manifold.  Consequently, the authors only consider a certain window within the supercritical regime.  The subsequent study \cite{BobrowskiMukherjee} includes the thermodynamic regime where they provide upper and lower bounds of the same order of growth for each Betti number.

Yogeshwaran, Subag, and Adler \cite{YSA} established limit laws (including a central limit theorem) in the thermodynamic regime for Betti numbers of random geometric complexes built over Poisson point processes in Euclidean space.
More recently, Goel, Trinh, and Tsunoda \cite{GoTrTs} established a limit law in the thermodynamic regime for Betti numbers
of random geometric complexes built over (possibly inhomogeneous) Poisson point processes in Euclidean space, where they also addressed the case when the point process is supported on a submanifold.

Hiraoka, Shirai, and Trinh \cite{HiShTr} proved a limit law for so-called ``persistent'' Betti numbers.  Although this goes in a rather separate direction motivated by topological data analysis, the formalism they use for describing the convergence of the persistence diagram has a loose resemblance to the setup of the current paper in that they introduce a sequence of random measures and show that it  converges in an appropriate sense to a deterministic measure.

A survey of other results on random 
geometric complexes is provided in \cite{BobrowskiKahle}.
Most progress in this area has been made only recently, but the problem of studying the topology of a random geometric complex (or equivalently the $\e$-neighborhood of a random point cloud) can be traced back to one of Arnold's problems (see the historical note at the end of the introduction).

A novelty of the current paper is that, whereas previous studies of random geometric complexes have focused on Betti numbers, we consider enumeration of connected components according to homotopy type,
a count that provides more refined topological information.

\subsection{The Riemannian case} Let $(M, g)$ be a compact Riemannian manifold of dimension $\dim(M)=d$, with normalized volume form $\Vol(M)=1$. Let $U_n=\{p_1, \ldots, p_n\}$ be a random set of points independently sampled from the uniform distribution on $M$. We denote by $\B(x, r)$ the Riemannian ball\footnote{In this paper we adopt the convention that when an object is denoted with a ``hat'' sign, then it is related to $\M.$ Analogous objects related to Euclidean space will have no ``hat''. For example a ball in $\M$ is denoted by $\B(x, r)$ and a ball in $\R^d$ by $B(x, r)$.} centered at $x\in \M$ of radius $r>0$. We fix a positive number $\alpha>0$ and build the random set:
\be\label{eq:uniform} \U_n=\bigcup_{k=1}^n \B(p_k, \alpha n^{-1/d}).\ee
We denote by $\check{C}(\U_n)$ the corresponding Cech complex (which for $n>0$ large enough, is homotopy equivalent to $\U_n$ itself, see Lemma \ref{lemma:convex} below).

Let now $\hat{\G}$ be the set of equivalence classes of $M$-geometric, connected simplicial complexes, up to homotopy equivalence (observe that this is a countable set). In other words, $\hat{\G}$ consists of all the connected simplicial complexes that arise as Cech complexes of some finite family of balls in $\M$. Note that different manifolds give rise to different sets $\hat{\G}$. For example, among all $\R^d$-geometric complexes we cannot find complexes with nonzero $d$-th Betti number; but if $M=S^d$, such complexes belong to $\hat{\G}$. When $M=\R^d$ we simply denote this set by $\G$.

\begin{figure}
\includegraphics[width=0.6\textwidth]{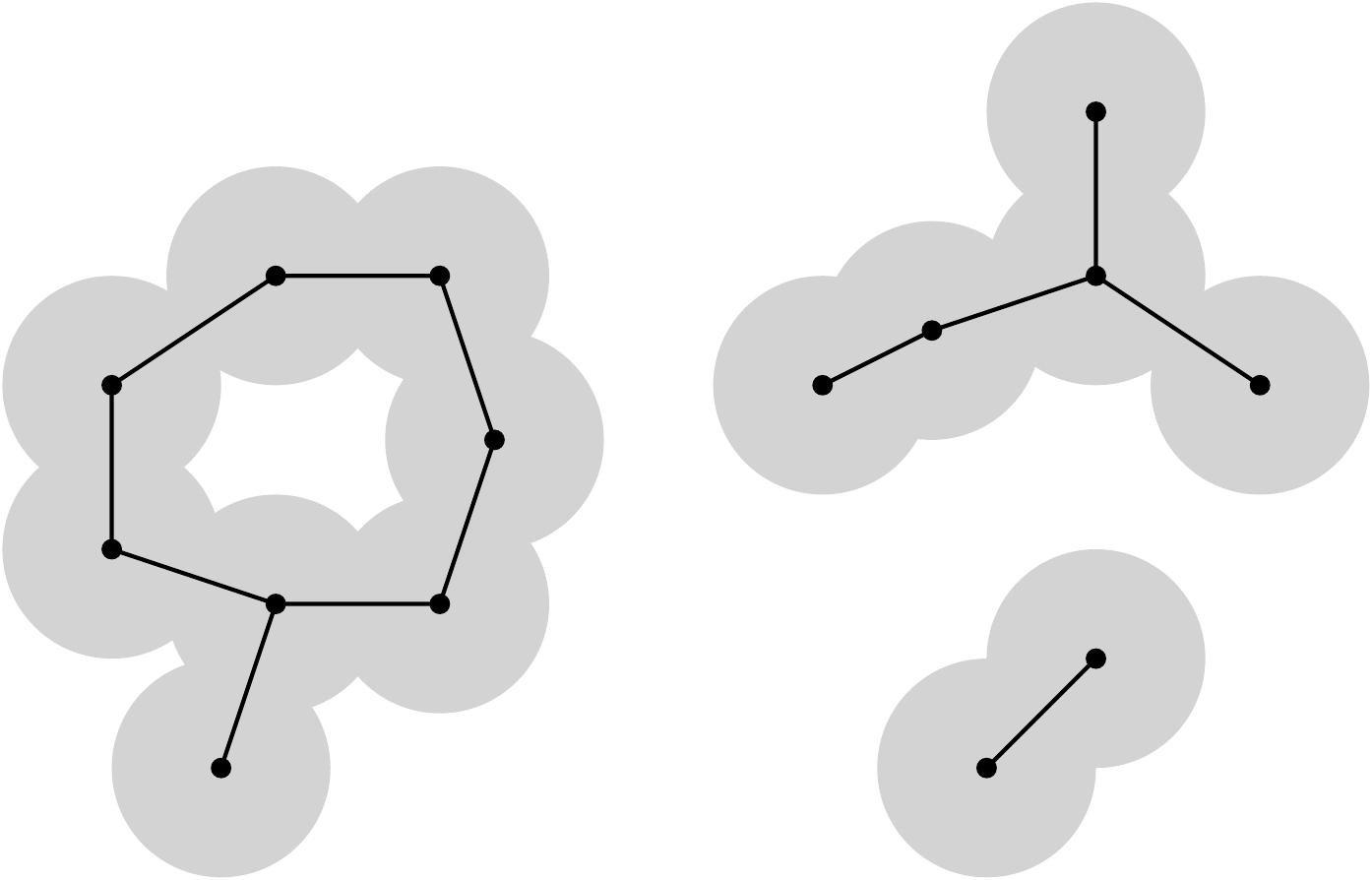}
	\caption{A geometric complex in the plane which is homotopy equivalent to $S^1\cup\{p_1\}\cup \{p_2\}$. The corresponding measure on the set of homotopy classes of connected, simplicial complexes is:
   $ \hat{\mu}=\frac{1}{3}\left(\delta_{[S^1]}+2\delta_{[\textrm{pt}]}\right).$ Theorem \ref{thm:main1} says that as $n\to \infty$ the random measure $\hat{\mu}_n$ converges to a deterministic measure.}\label{fig:complex}
\end{figure}
Given $\U_n$ as above, we define the random probability measure $\hat\mu_n$ on $\hat\G$:
\be \hat{\mu}_n= \frac{1}{b_0(\check{C}(\U_n))} \sum \delta_{[s]}, \ee
where the sum is over all connected components $s$ of $\U_n$, $[s]$ denotes the type of $s$ (i.e., the equivalence class of all connected complexes homotopy equivalent to $s$), 
and $b_0$ denotes the number of connected components. 

\begin{remark}The next theorem deals with the convergence of the random measure $\hat\mu_n$ in the limit $n\to\infty.$ We endow the set $\mathscr{P}$ of probability measures on the countable set $\hat\G$ with the \emph{total variation distance}:
\be d(\mu_1, \mu_2)=\sup_{A\subset \hat{\G}}\left|\mu_1(A)-\mu_2(A)\right|.\ee
In this way $\hat\mu_n$ is a random variable with values in the metric space $(\mathscr{P}, d)$. Convergence \emph{in probability} (which is used in Theorem \ref{thm:main1} and Theorem \ref{thm:mainpoisson}) of a sequence of random variables $\{\mu_n\}_{n\in\mathbb{N}}$ to a limit $\mu$ means that for every $\e>0$ we have $\lim_{n \rightarrow \infty} \PP\{d(\mu_n, \mu)>\e\}=0.$
\end{remark}

\begin{thm}\label{thm:main1}
The random measure $\hat\mu_n$ converges in probability to a universal deterministic probability measure $\mu\in \mathscr{P}$ supported on the set $\G$ of connected $\R^d$-geometric complexes.
\end{thm}
The ``universal'' in the previous statement means that $\mu$ does not depend on the manifold $M$ (but it depends on its dimension $d$ and on the parameter $\alpha$).

\begin{remark}
Since $\G$ is a proper subset of $\hG$,
the measure $\mu$ 
does not charge some points in $\hG$.
This is consistent with the 
findings of \cite{BobrowskiWeinberger}
where it was shown that an additional factor of $\log n$
is needed in the radii of the balls defining $\U_n$
in order to see the so-called ``connectivity threshold'' where nontrivial $d$-dimensional homology appears.
\end{remark}

\begin{remark}In the one-dimensional case $d=1$, the set $\G$ contains only one element: the class of the point (since any connected geometric complex in $\R$ is contractible). The case $d=2$ is already more interesting, since in this case $\G=\{[w_k]\}_{k\in \mathbb{N}}$ where $w_k$ is the wedge of $k$-circles ($k=0$ is the point). In general the support of $\mu$ is more difficult to describe.
\end{remark}
\begin{remark}We can write the limiting measure $\mu$ as:
\be
\mu=\sum_{\gamma\in \G} a_{\gamma}\delta_{\gamma}
\ee
for some non-negative constants $a_\gamma$, $\gamma\in \G$, which depend on the $\alpha>0$ appearing in \eqref{eq:uniform},
and satisfy $a_\gamma = c_\gamma/c$ with $c_\gamma,c$ defined in Proposition \ref{prop:ergodic}.
All of the coefficients $a_\gamma$ are strictly positive by Proposition \ref{prop:alleuclid}.
\end{remark}

The following result is related to the positivity of all coefficients $a_\gamma$.
While it is not needed for showing such positivity (which follows from Proposition \ref{prop:alleuclid}), it provides additional information on the prevalence of localized components with prescribed homotopy type throughout the manifold, see Section \ref{sec:quant}.

\begin{prop}[Existence of all topologies]\label{prop:all}Let $\Pc_0\subset \R^d$ be a finite geometric complex and  $\alpha>0$. There exist $R, a>0$ (depending on $\Pc_0$ and $\alpha$ but
independent of $M$ and $n$) such that for every $p\in \M$ and for $n$ large enough:
\be \PP\left\{\U_n\cap \B(p, Rn^{-1/d})\simeq \Pc_0\right\}>a.\ee
\end{prop}

\begin{remark}
Let us point out an interesting consequence of the previous Proposition \ref{prop:all}: given a compact, embedded manifold $M_0 \hookrightarrow \R^d$, then for $R>0$ large enough with positive probability the pair $(\R^d, M_0)$ is homotopy equivalent to the pair $(\B(p, R n^{-1/d}), \U_n\cap\B(p, R n^{-1/d}))$. This follows from the fact that, by \cite[Proposition 3.1]{NSW}, one can cover $M_0$ with (possibly many) small Euclidean balls $M_0\subset \bigcup_{k=1}^\ell B(p_k, \e)=\U$ with the inclusion $M_0\hookrightarrow \U$ a homotopy equivalence --  hence the pair $(\R^d, M_0)$ is homotopy equivalent to a pair $(\R^d, \Pc_0)$ with $\Pc_0$ a $\R^d$-geometric complex. 
\end{remark}

\subsection{The local model (the Euclidean case)}The proof of Theorem \ref{thm:main1} for the Riemannian case involves a study of a rescaled version of the problem in a small neighborhood of a given point. Specifically, one can fix  $R>0$ and a point $p\in M$ and study the asymptotic structure of our random complex in the ball $\hat{B}(p, Rn^{-1/d})$. The random geometric complex that we obtain in the $n\to \infty$ limit can be described as follows.

Let $P=\{p_1, p_2, \ldots\}$ be a set of points sampled from the standard spatial Poisson distribution on $\R^d$ 
and for $\alpha>0$ consider the random set:
\be \Pc=\bigcup_{p\in P}B(p, \alpha).\ee
We also define $\Pc_R$ to
be the subset of $\Pc$ consisting
of all the connected components of $\Pc$ 
that are completely contained in the interior of $B(0,R)$.
Note that each $B(p, \alpha)$ is now convex, 
and, by the Nerve Lemma, $\Pc_R$ is homotopy equivalent to the simplicial complex $\check{C}(\Pc_R)$. 
The relation between $\mathcal{U}_n\cap \hat{B}(p, R n^{-1/d})$ and $\mathcal{P}_R$ is described in Theorem \ref{thm:coupling}.

Similarly to what we have done above, we define the random probability measure $\mu_R$ on the set $\G$ of homotopy types of finite and connected $\R^d$-geometric complexes:
\be \mu_R=\frac{1}{b_0(\check{C}(\Pc_R))} \sum \delta_{[s]}, \ee
where the sum is over all connected components 
$s$ of $\Pc_R$.
The following result provides a limit law
for $\mu_R$.

\begin{thm}\label{thm:mainpoisson}
The family of random measures $\mu_R$ converges in probability to a deterministic probability measure $\mu\in \mathscr{P}$ whose support is all of $\G$. 
\end{thm}
It is important to note that the limiting measure $\mu$ 
appearing in Theorem \ref{thm:mainpoisson} is the same one appearing in Theorem \ref{thm:main1} (this explains the statement on the support of the limiting measure in Theorem \ref{thm:main1}).

Besides their positivity,
little is known about the coefficients $a_\gamma$ in $\mu$,
and a worthwhile computational problem would be to perform Monte Carlo simulations in order to estimate their numerical values
and how they depend on $\alpha$.
Concerning dependence on $\alpha$, a direction that has been suggested to us by Matthew Kahle is to study whether the dependence of $\mu$ on $\alpha$ exhibits any interesting behavior related to the ``percolation threshold''
(recalling that our random geometric complex is associated to continuum percolation with disks for which existence of a percolation threshold  $\alpha = \alpha_c$ is known \cite{MeesterRoy}).

While this paper was under review,
K. A. Dowling and the third author posted a preprint \cite{DL}
further adapting these methods to study the limiting homotopy distribution for
random cubical complexes
associated to Bernoulli site percolation on a cubical grid,
where it was shown that the limiting homotopy measure has an exponentially decaying tail for subcritical percolation
and a subexponential tail (slower than exponential decay) for supercritical percolation.  It is then natural to pose a specific version of the above problem suggested by M. Kahle, namely, to investigate the tail decay of $\mu$ in the current setting of random geometric graphs and to determine whether it exhibits a phase transition at the percolation threshold $\alpha=\alpha_c$ (see \cite[Concluding Remarks]{DL}).

\noindent {\bf Outline of the paper.}
We prove Theorem \ref{thm:mainpoisson}
addressing the Euclidean setting in Section \ref{sec:Euclid}.
In Section \ref{sec:semilocal}, we establish
the ``semi-local'' result involving a double-scaling limit within a neighborhood on the manifold,
and in Section \ref{sec:global}
we collect the semi-local information throughout
the manifold
in order to prove the global result Theorem \ref{thm:main1} for the manifold setting.
We prove Proposition \ref{prop:all} 
in Section \ref{sec:quant}.
Section \ref{sec:prelim} contains some basic tools
used throughout the paper,
including the integral geometry sandwiches that play an essential role.

\noindent {\bf Historical Note.}
The study of the topology of random simplicial complexes has taken shape only recently with intense activity in the past few years, but it is worth mentioning (as it seems to have been forgotten) that this theme was proposed by V.I. Arnold in the early 1970s,
with specific attention given to random geometric complexes in the thermodynamic regime.
In the collection \cite{Arnold} of Arnold's problems, the 28th problem from 1973 states
(notice that the set considered is homotopy equivalent to a geometric complex by the nerve lemma):

\emph{Consider a random set of points in $\R^d$
with density $\lambda$.
Let $V(\alpha)$ be the $\alpha$-neighborhood of this set. Consider the averaged Betti numbers
$$\beta_i(\alpha,\lambda) := \lim_{R \rightarrow \infty} \frac{b_i(V(\alpha) \cap B(0,R))}{R^d}.$$
Investigate these numbers.}

\subsection*{Acknowledgements}
This work was initiated 
during the conference ``Stochastic Topology and Thermodynamic Limits'' that was hosted at ICERM, and part of the work was completed during   
a second week-long visit to ICERM through the collaborate@ICERM program. The authors wish to thank the institute for their support and for a pleasant and hospitable work environment. The authors would also like to thank the anonymous referee for a careful reading of the paper and many helpful comments regarding the exposition.  This research was conducted while A.A. was supported by NSF Grant CAREER DMS-1653552.

\section{Preliminary material}\label{sec:prelim}

In this section we collect some basic tools used throughout the paper.
\subsection{Geometry}
A subset $A$ of a Riemannian manifold $(M, g)$ is called \emph{strongly convex} if for any pair of points $y_1, y_2\in \textrm{clos}(A)$ there exists a unique minimizing geodesic joining these two points such that its interior is entirely contained in $A$ (see \cite{CE,DoCarmo}). 
\begin{lemma}\label{lemma:convex}Let $(M, g)$ be a compact Riemannian manifold. There exists $r_0>0$ such that for every point $x\in M$ and every $r<r_0$ the ball $\B(x, r)$ is strongly convex and contractible. Moreover for every $x_1, \ldots, x_k\in \M$ and $0<r_1, \ldots, r_k<r_0$ the set $\bigcap_{j=1}^k\B(x_j, r_j)$ is also strongly convex and contractible. In particular, by the Nerve Lemma,  the set $\bigcup_{j=1}^k\B(x_j, r_j)$ is homotopy equivalent to its associated Cech complex.
\end{lemma}

\begin{proof}By \cite[Theorem 5.14]{CE} there exists a positive and continuous function $r:M\to (0, \infty)$ such that if $r<r(x)$, then $\B(x, r)$ is strictly convex (this is in fact due to Whitehead). Since $\M$ is compact, then $r_0=\min r>0$. Any strongly convex set in a Riemannian manifold is contractible with respect to any of its point (star-shaped in exponential coordinates), hence it follows that for $r<r_0$ the ball $\B(x, r)$ is also contractible. To finish the proof, we simply observe that the intersection of strongly convex sets $A_1, A_2$ is still strongly convex: in fact given two points $y_1, y_2\in A_1\cap A_2$, by strong convexity of the sets, the unique minimizing geodesic joining the two points is contained in both sets.
\end{proof}

From now on, the notation 
$\genfrac{\{ }{\}}{0pt}{}{\ell}{k}$ denotes
the collection of all $k$-element subsets
of $\{1,2,...,\ell\}.$

We will say that a $\R^d$-geometric complex $\bigcup_{k=1}^\ell B(y_k, r)$ is \emph{nondegenerate} if for every $1\leq k\leq \ell$ and $J=\{j_1, \ldots, j_k\}\in \genfrac{\{ }{\}}{0pt}{}{\ell}{k}$ the intersection $\bigcap_{j\in J}\partial B(y_j, r)$ is transversal (in particular this intersection is empty for $k>d$).

Random geometric complexes are nondegenerate with probability one. However, it could be that without the nondegeneracy assumption one could construct a geometric complex which is not homotopy equivalent to any nondgenerate one. This is not the case, as next Lemma shows.
\begin{lemma}\label{rmk:nondeg} The set of homotopy types of $\R^d$-geometric, connected, nondegenerate complexes coincides with $\G$\footnote{Recall that we did not assume the nondegeneracy condition in the definition of $\G$.}.
\end{lemma}
\begin{proof}Given  a possibly degenerate $\Pc=\bigcup_{k=1}^{\ell}B(y_k,r)$, let $f:\R^d\to \R$ be the semialgebraic and continuous function defined by
\be f(x)=d(x, \{y_1, \ldots, y_k\})=\min_{k}\|y_k-x\|,\ee and observe that:
\be \bigcup_{k=1}^{\ell}B(y_k,r)=\{f\leq r\}.\ee
We consider now the semialgebraic, monotone family $\{X(r+\epsilon)=\{f\leq r+\epsilon\}\}_{\epsilon\geq 0}$. By \cite[Lemma 16.17]{BPR} for $\epsilon>0$ the inclusion $X(r)\hookrightarrow X(r+\epsilon)$ is a homotopy equivalence. It suffices therefore to show that for $\epsilon>0$ small enough $X(r+\epsilon)$ is nondegenerate; this follows from the fact that given points $y_1, \ldots, y_\ell\in \R^d$, for every $1\leq k\leq d$ and $J=\{j_1, \ldots, j_k\}\in \genfrac{\{ }{\}}{0pt}{}{\ell}{k}$ there are only finitely many $r>0$ such that the intersection $\bigcap_{j\in J}\partial B(y_j, r)$ is nontransversal (and the number of possible multi-indices to consider is also finite). 
\end{proof}
The following Proposition plays an important role in all asymptotic stability arguments. 
\begin{prop}\label{prop:stability} Let $(\M, g)$ be a compact Riemannian manifold of dimension $d$ and $p\in \M$.
Let $\Pc\subset \R^d$ be a nondegenerate complex such that:
\be\Pc=\bigcup_{j=1}^\ell B(y_j, r)\subset B(0, R')\ee
for some points $y_1, \ldots, y_\ell \in \R^d$ and $r, R'>0$. Given $\alpha>0$ set $R=\frac{\alpha R'}{r}$ and consider the sequence of maps:
\be \psi_n:\B(p, Rn^{-1/d})\xrightarrow{\mathrm{exp}_p^{-1}}B_{T_pM}(0, Rn^{-1/d})\xrightarrow{\frac{r}{\alpha}n^{1/d}}B_{T_pM}(0, R')\simeq B(0, R'). \ee
Denoting by $\varphi_n$ the inverse of $\psi_n$, there exist $\e_0>0$ and $n_0>0$ such that if $\|\tilde{y}_k-y_k\|\leq \e_0$ for every $k=1, \ldots,\ell$ then for $n\geq n_0$ we have:
\be\bigcup_{k=1}^\ell\B(\varphi_n(\tilde y_k), \alpha n^{-1/d})\simeq \bigcup_{k=1}^\ell B(y_k, r). \ee
\end{prop}
\begin{proof}

For $k\leq d$ and for every $J=\{j_1, \ldots, j_k\}\in \genfrac{\{ }{\}}{0pt}{}{\ell}{k}$ either one of these possibilities can verify:
\begin{enumerate}
\item $\bigcap_{j\in J}B(y_j, r)\neq \emptyset$, in which case, by nondegeneracy, there exists $\e_J$ and $y_J$ such that $\|y_J-y_j\|<r-\e_J$ for all $j\in J$;
\item $\bigcap_{j\in J}B(y_j, r)= \emptyset$, in which case there is no $y$ solving $\|y-y_j\|\leq r$ for all $j\in J$.
\end{enumerate}

Since the sequence of maps $d_n:B(0, R')\times B(0, R')\to \R$ defined by
\be\frac{rn^{1/d}}{\alpha}\cdot d_n(z_1, z_2)=d_{\M}(\varphi_n(z_1), \varphi_n(z_2))\ee
converges uniformly to $d_{\R^d}$, then for every $\delta>0$ there exists $n_1>0$ such that for all pairs of points $z_1, z_2\in B(0, R')$ and for all $n\geq n_1$ we have:
\be\label{eq:delta} \left|\frac{rn^{1/d}}{\alpha}\cdot d_\M(\varphi_n(z_1), \varphi_n(z_2))-\|z_1-z_2\|\right|\leq \delta.\ee
For every index set $J$ satisfying condition (1) above, choosing $\delta=\frac{\e_{J} r}{3\alpha}$ and setting $\e_J=\delta$, the previous inequality \eqref{eq:delta} implies that, if $\|\tilde{y}_k-y_k\|<\e_J$ for every $k=1, \ldots, \ell$, then for $n\geq n_J$:
\be d_\M(\varphi_n(\tilde y_j), \varphi_n(y_J))<\alpha n^{-1/d}.\ee
This means that the combinatorics of the covers $\{B(y_j, r)\}_{j\in J}$ and $\{\B(\varphi_n(\tilde y_j), \alpha n^{-1/d})\}_{j\in J}$ are the same if $\|y_j-\tilde{y}_j\|<\e_J$ for $j\in J$ and $n\geq n_J$.

Let us consider now an index set $J$ satisfying condition (2) above. We want to prove that there exists $\e_J>0$ and $n_J$ such that if $\|\tilde{y}_j-y_j\|<\e_J$ for all $j\in J$, then for $n\geq n_J$ the intersection $\cap_{j\in J}\B(\varphi_n(\tilde y_j), \alpha n^{-1/d})$ is still empty. We argue by contradiction and assume there exist a sequence of points $x_n\in \B(p, Rn^{-1/d})$ and for $j\in J$ points $y_{j,n}\in B(0, R')$ with $\|y_{i,n}-y_j\|\leq \frac{1}{n}$ such that for all $j\in J$ and all $n$ large enough:
\be \label{eq:ineq}d_\M(x_n, \varphi_n(y_{j,n}))<\alpha n^{-1/d}.\ee
We call $y_n=\psi_n(x_n)$ and assume that (up to subsequences) it converges to some $\overline {y}\in B(0, R')$.
Using again the uniform convergence of $d_n$ to $d_{\R^d}$, the inequality \eqref{eq:ineq} would give:
\be r>\lim_{n\to \infty}\frac{n^{1/d} r}{\alpha}\cdot d_\M(x_n, \varphi_n(y_{j,n}))=\|\overline{y}-y_j\|\quad \forall j\in J\ee
which gives the contradiction $\overline{y}\in \bigcap_{j\in J}B(y_j, r)= \emptyset$.

Set now $n_1=\max_{J\in \genfrac{\{ }{\}}{0pt}{}{\ell}{k}, k\leq d}n_J$ and $\e_0=\min_{J\in \genfrac{\{ }{\}}{0pt}{}{\ell}{k}, k\leq d}\e_J$. We have proved that, if $\|\tilde{y}_j-y_j\|<\e_0$ for all $j=1, \ldots, \ell$, then for all $n\geq n_1$ the two open covers $\{B(y_j, r)\}_{j\in J}$ and $\{\B(\varphi_n(\tilde y_j), \alpha n^{-1/d})\}_{j\in J}$ have the same combinatorics. In particular their Cech complex is the same. Moreover, Lemma \ref{lemma:convex} implies that for a possibly larger $n_0\geq n_1$ all the balls $\B(x, \alpha n^{-1/d})$ are strictly convex in $\M$; consequently, by the Nerve Lemma, for $n$ larger than such $n_0$ these two open covers are each one homotopy equivalent to their Cech complexes, hence they are themselves homotopy equivalent.
\end{proof}

\subsection{Measure theory}

The following lemma will be used in
the proof of Theorem \ref{thm:mainpoisson}.
This lemma and its proof are essentially in \cite[Thm. 4.2 (2)]{SarnakWigman}, but we provide
a proof to make the paper more self-contained and to ensure that it is clear this result is purely measure-theoretic.

\begin{lemma}\label{lemma:cvg}
	Let $\mu_{\lambda} = \sum a_{\lambda,k} \delta_k$ be a one-parameter family of random probability measures on $\NN$, and let $\mu = \sum a_k \delta_k$ be a deterministic probability measure on $\NN$.	Assume that for every $k \in \NN$ $a_{\lambda,k} \rightarrow a_k$ in probability as $\lambda \rightarrow \infty$. Then $ \mu_\lambda \rightarrow \mu $ in probability,
i.e., for every $\varepsilon >0$ we have
	$$\lim_{\lambda \rightarrow \infty} \PP \{ d(\mu_\lambda,\mu) \geq \e \} = 0,$$
where $d$ denotes the total variation distance.
\end{lemma}

\begin{proof}
Let $\delta > 0$ be arbitrary.

Since $\mu$ is a probability measure on $\NN$,
there exists $K$ such that
\begin{equation}\label{eq:tail}
\sum_{k \geq K} a_k < \frac{\e}{4} .
\end{equation}

We have
$$\PP \left\{ |a_{\lambda,k} - a_k| > \frac{\e}{4K} \right\} < \frac{\delta}{2K},$$
which implies (by a union bound)
\begin{equation}\label{eq:abs}
\PP \left\{  \sum_{k < K} \left| a_{\lambda,k} - a_k \right| > \frac{\e}{4} \right\} < \frac{\delta}{2},
\end{equation}
and also (by the triangle inequality)
\begin{equation}\label{eq:triangle}
 \PP \left\{ \left| \sum_{k < K} a_{\lambda,k} - \sum_{k < K} a_k \right| > \frac{\e}{4} \right\} < \frac{\delta}{2},
\end{equation}
for $\lambda \geq \lambda_0$.

The estimate (\refeq{eq:triangle}) 
implies an estimate for the tails:
\begin{equation}\label{eq:abstail}
\PP \left\{ \left| \sum_{k \geq K} a_{\lambda,k} - \sum_{k \geq K} a_k \right| > \frac{\e}{4} \right\} < \frac{\delta}{2},
\end{equation}
since
$$ \left| \sum_{k < K} a_{\lambda,k} - \sum_{k < K} a_k \right| = \left| \sum_{k \geq K} a_{\lambda,k} - \sum_{k \geq K} a_k \right| ,$$
which follows from $\mu_\lambda$ and $\mu$ being probability measures.

For any $\lambda > \lambda_0$, we then have
\begin{equation}\label{eq:arn}
\PP \left\{ \sum_{k \geq K} a_{\lambda,k} > \frac{\e}{2} \right\} < \frac{\delta}{2}.
\end{equation}
Indeed, if 
$$ \sum_{k \geq K} a_{\lambda,k} > \frac{\e}{2}$$
then equation (\refeq{eq:tail}) gives
$$\left| \sum_{k < K} a_{\lambda,k} - \sum_{k < K} a_k \right| > \frac{\e}{4},$$
and (\refeq{eq:arn}) then follows from (\refeq{eq:abstail}).

In order to estimate the total variation distance between $\mu_\lambda$ and $\mu$, 
let $A \subset \NN$ be arbitrary. 
We have:
\begin{align*}
\left| \sum_{k \in A} a_{\lambda,k} - \sum_{k \in A} a_k \right| &=
  \left| \sum_{k \in A, k < K } (a_{\lambda,k} - a_k) + \sum_{k \in A, k \geq K} a_{\lambda,k} - \sum_{k \in A, k \geq K} a_{k} \right| \\
  &\leq   \sum_{k \in A, k < K }  \left|a_{\lambda,k} - a_k\right| +  \sum_{k \in A, k \geq K} a_{\lambda,k} + \sum_{k \in A, k \geq K} a_{k} \\
  & \leq \sum_{k < K } \left| a_{\lambda,k} - a_k \right| +  \sum_{k \geq K} a_{\lambda,k} + \sum_{k \geq K} a_{k} \\
  & \leq \sum_{k < K } \left| a_{\lambda,k} - a_k \right| +  \sum_{k \geq K} a_{\lambda,k} + \frac{\e}{4}.
\end{align*}

Using a union bound,
this implies
$$\PP \left\{ \left| \sum_{k \in A} a_{\lambda,k} - \sum_{k \in A} a_k \right| > \e \right\}
\leq \PP \left\{ \sum_{k < K} \left| a_{\lambda,k} - a_k \right| > \frac{\e}{4} \right\} + \PP \left\{ \sum_{k \geq K} a_{\lambda,k} > \frac{\e}{2} \right\},$$
which is less than $\delta$ by (\refeq{eq:abs}) and (\refeq{eq:arn}).

This implies that for every $\delta > 0$
we have, for all $\lambda$ sufficiently large,
$$\PP \left\{ \sup_{A \subset \NN} \left| \mu_\lambda (A) - \mu(A) \right| \geq \e \right\}
\leq \delta,$$
i.e. we have shown
$$\lim_{\lambda \rightarrow \infty} \PP \left\{ d(\mu_\lambda,\mu) \geq \e \right\} = 0.$$
\end{proof}

\subsection{The ergodic theorem}
The proof of Proposition \ref{prop:ergodic}
uses the following special case of
the $d$-dimensional ergodic theorem.
We follow \cite[Ch. 2]{MeesterRoy}
and \cite[Sec. 12.2]{DV-J}).

\begin{thm}[Ergodic Theorem]\label{thm:ergodic}
Let $(\Omega,\mathcal{F},\rho)$
be a probability space, 
and let $T_x$, $x \in \R^d$
be an $\R^d$-action on $\Omega$.
Let $f \in L^1(\rho)$,
and suppose further that 
the action of $T_x$ on $\Omega$ is ergodic.
Then we have
$$\frac{1}{\Vol(B_R)} \int_{B_R} f(T_x(\omega)) dx \rightarrow \EE \, f(\omega) \text{ a.s. and in } L^1 $$
as $R \rightarrow \infty$.
\end{thm}

Let us explain the terminology appearing in the statement of this theorem.  An \emph{$\R^d$-action}
$T_x$, $x \in \R^d$
is a group of invertible, commuting, measure-preserving transformations
acting measurably on 
a probability space $(\Omega,\mathcal{F},\rho)$ and indexed by $\R^d$.
An $\R^d$-action $T_x$ is said to be \emph{ergodic}
if any invariant event has probability
either zero or one.

For our application of the ergodic theorem (see the proof of Proposition \ref{prop:ergodic} below), 
the $\R^d$-action
$T_x$ will simply be translation by $x$ 
acting on the Poisson process
(this case is known to be ergodic \cite[Prop. 2.6]{MeesterRoy}).

    \subsection{Component counting function and the integral geometry sandwiches}

\begin{defi}[Component counting function]\label{def:ccf}
Let $Y_1, Y_2 \subset X$ and $Z$ be topological spaces (in the case of our interest they will be homotopy equivalent to finite simplicial complexes). We denote by $\N(Y_1, Y_2;[Z])$ the number of connected components of $Y_1$ \emph{entirely contained} in the interior of $Y_2$ and which have the same homotopy type as $Z$. Similarly, we denote by $\N^*(Y_1, Y_2;[Z])$ the number of connected components of $Y_1$ which \emph{intersect} $Y_2$ and which have the same homotopy type as $Z$.
\end{defi}

    \begin{thm}[Integral Geometry Sandwich]\label{thm:IGS}

Let $\Pc$ be a generic geometric complex in $\R^d$ 
and fix $\gamma \in \G$.
Then for $0<r<R$
\be 
\int_{B_{R-r}} \frac{\N(\Pc, B(x, r); \gamma)}{\Vol \left( B_r \right)} dx
\leq \N(\Pc, B_R; \gamma) \leq \int_{B_{R+r}} \frac{\N^*(\Pc, B(x, r); \gamma)}{\Vol \left( B_r \right)} dx.
\ee 
\end{thm}
\begin{thm}[Integral Geometry Sandwich on a Riemannian manifold]\label{thm:IGSmanif}
Let $\U$ be a generic geometric complex on $\M$
and fix $\gamma \in \G$. Then for any $\e>0$
there exists $\eta >0$ such that for every $r < \eta$
\be 
(1-\e) \int_{\M} \frac{\N(\U, \hB(x,r);\gamma)}{\Vol \left( B_r \right)} dx
\leq \N(\U,M; \gamma) \leq (1+\e) \int_{\M} \frac{\N^*(\U, \hB(x,r); \gamma)}{\Vol \left( B_r \right)} dx,
\ee 
where $B_r$ still denotes the Euclidean ball of radius $r$.
\end{thm}

\begin{proof}[Proofs of Theorems \ref{thm:IGS} and \ref{thm:IGSmanif}]
These results follow from the same proof
as in \cite{SarnakWigman}.
\end{proof}

\begin{remark}\label{rmk:IGS}
Similar statements hold true if we take the sum over \emph{all} components, ignoring their type (an observation used throughout the paper). 
More precisely, denoting by $\N(Y_1, Y_2)$ the number of components of $Y_1$ entirely contained in the interior of $Y_2$ and by $\N^*(Y_1, Y_2)$ the number of components of $Y_1$ that intersect $Y_2$, we have the following inequality:
\be 
\int_{B_{R-r}} \frac{\N(\Pc, B(x, r))}{\Vol \left( B_r \right)} dx
\leq \N(\Pc, B_R) \leq \int_{B_{R+r}} \frac{\N^*(\Pc, B(x, r))}{\Vol \left( B_r \right)} dx
\ee 
 and, in the Riemannian framework:
 
 \be 
(1-\e) \int_{\M} \frac{\N(\U, \hB(x,r))}{\Vol \left( B_r \right)} dx
\leq \N(\U,M) \leq (1+\e) \int_{\M} \frac{\N^*(\U, \hB(x,r))}{\Vol \left( B_r \right)} dx.
\ee 
 Since both $\Pc$ and $\U$ have only finitely many components, these inequalities follow by simply summing up the two inequalities from the previous theorems over all components type (the sums are over finitely many elements).
In fact the integral geometry sandwiches as proved in \cite{SarnakWigman} are adaptations of the original construction from \cite{NazarovSodin}, where components were counted without regard to topological type.
\end{remark}

\section{Limit law for the Euclidean case}\label{sec:Euclid}

In this section we prove 
Theorem \ref{thm:mainpoisson}.
The main step is provided by the following proposition.
We continue to use the above notation for the component counting function (see Definition \ref{def:ccf}).

\begin{prop}\label{prop:ergodic}
For every homotopy type $\gamma\in \G$ there exists a constant $c_\gamma$ such that the random variable \be c_{R, \gamma}=\frac{\N(\Pc, B(0,R);\gamma)}{\Vol(B(0,R))}\ee
converges to a constant $c_\gamma$ in $L^1$ as $R\to \infty$. The same is true for the random variable 
\be c_R=\frac{\N(\Pc, B(0,R))}{\Vol(B(0,R))},\ee (i.e. when we consider all components, with no restriction on their types): as $R\to \infty$, it converges to a constant $c$ in $L^1$.
\end{prop}

The next proposition is proved in Section \ref{sec:quant}.

\begin{prop}\label{prop:alleuclid}
The constants
$c_\gamma$ defined in
Proposition \ref{prop:ergodic}
are positive for all $\gamma \in \G$.
\end{prop}

\begin{proof}[Proof of Proposition \ref{prop:ergodic}] 
The proof follows the argument from \cite[Theorem 3.3]{SarnakWigman}, with some needed modifications.

We will use the shortened notation $\N_R=\N(\Pc, B(0,R);\gamma)$, $\N(x, r)=\N(\Pc, B(x, r);\gamma)$ and $\N^*(x, r)=\N^*(\Pc, B(x, r);\gamma)$ ($\gamma$ will be fixed for the rest of the proof and we omit dependence on it in the notation).
Using Theorem \ref{thm:IGS} we can write, for $0<\alpha<r<R$:
\begin{equation} \label{eq:starting}
\left(1-\frac{r}{R}\right)^d\frac{1}{\Vol(B_{R-r})}\int_{B_{R-r}} \frac{\N(x,r)}{\Vol \left( B_r \right)} dx
\leq \frac{\N_R}{\Vol(B_R)} \leq \left(1+\frac{r}{R}\right)^d\frac{1}{\Vol(B_{R+r})}\int_{B_{R+r}} \frac{\N^*(x, r)}{\Vol \left( B_r \right)} dx.
\end{equation} 

Denoting by $A(x, r, \alpha)$ the annulus $\{z\in \R^d: r-\alpha\leq\|x-z\|\leq r+\alpha \}$, we can estimate the integral on the r.h.s. of \eqref{eq:starting} with:
\be\label{eq:N*est}
\int_{B_{R+r}} \frac{\N^*(x, r)}{\Vol \left( B_r \right)} dx\leq \int_{B_{R+r}} \frac{\N(x, r)}{\Vol \left( B_r \right)} +\frac{\#P\cap A(x, r, \alpha) }{\Vol \left( B_r \right)}dx .\ee
In fact, if a component of $\Pc$ is not entirely contained in the interior of $B(x, r)$, then it touches the boundary of $B(x, r)$ and hence this component must contain a point $p\in P \cap A(x, r, \alpha)$. 

We now take $R \rightarrow \infty$
(with $r$ fixed)
and use the Ergodic Theorem 
in order to assert that
\be\label{eq:ergodicappl}
\frac{1}{\Vol(B_{R-r})}\int_{B_{R-r}} \frac{\N(x,r)}{\Vol \left( B_r \right)} dx \rightarrow 
\lambda(r) \quad \text{a.s. and in } L^1,
\ee
as $R \rightarrow \infty$
where $
\lambda(r) := \EE \frac{\N(x,r)}{\Vol \left( B_r \right)}$ is a constant.

In order to apply the ergodic theorem (Theorem \ref{thm:ergodic} stated above) we introduce the function $f$ defined as
$$f(P) = \frac{\N(0,r)}{\Vol{(B(0,r)})},$$
where $r$ is fixed,
and the dependence of $f$ on the Poisson process $P$
is through $\Pc$ which we recall is 
the $\alpha$-neighborhood of $P$.
We also let $T_x$ denote translation by $x$ 
acting on the Poisson process $P$.  This action is ergodic as noted above in Section \ref{sec:prelim}.
We also have $f \in L^1$, since
$$\EE \left[ \frac{\N(\Pc,B(0,r),\gamma)}{\Vol(B(0,r))} \right] \leq  \frac{\EE \# P \cap B(0,r)}{\Vol(B(0,r))} = 1 < \infty,$$ so that the ergodic theorem may be applied.

Furthermore, we notice that
$$f(T_x(P)) = \frac{\N(x,r)}{\Vol{(B(0,r)})},$$
i.e., shifting $P$
has the same effect as recentering the ball $B(0,r)$
to $B(x,r)$.
Thus, the result of applying the ergodic theorem to this choice of $f$ gives precisely the convergence statement in \eqref{eq:ergodicappl}.

Note that the same convergence statement in \eqref{eq:ergodicappl} holds for $\frac{1}{\Vol(B_{R+r})} \int_{B_{R+r}} \frac{\N(x, r)}{\Vol \left( B_r \right)} dx$, namely,
\be\label{eq:ergodicappl2}
\frac{1}{\Vol(B_{R+r})}\int_{B_{R+r}} \frac{\N(x,r)}{\Vol \left( B_r \right)} dx \rightarrow 
\lambda(r) \quad \text{a.s. and in } L^1,
\ee
as $R \rightarrow \infty$
where $\lambda(r)$ is the same constant as in \eqref{eq:ergodicappl}.

We also have
\be\label{eq:LLN}
\frac{1}{\Vol(B_{R+r})}\int_{B_{R+r}} \frac{\#P\cap A(x, r, \alpha) }{\Vol \left( B_r \right)}dx \rightarrow a(r) \quad \text{a.s. and in } L^1
\ee
as $R \rightarrow \infty$
where
$a(r) := \EE \frac{\#P\cap A(x, r, \alpha) }{\Vol \left( B_r \right)} =  \frac{\Vol(A(x, r, \alpha)) }{\Vol \left( B_r \right)}$.
This follows from the ergodic theorem as well
(although it seems more natural to view it as a consequence of the law of large numbers).

Let $\e>0$ be arbitrary.
Since $a(r)=O(r^{-1})$
we can choose $r$ sufficiently large
that $a(r) < \e$.  We then choose 
$R>>r>>0$ sufficiently large so that
$(1 + r/R)^d < 1+\e$, $(1-r/R)^d > 1-\e$.
Using the above convergence statements \eqref{eq:ergodicappl},
\eqref{eq:ergodicappl2},
and \eqref{eq:LLN}
(while making $R$ even larger if necessary)
we have,
\be\label{eq:inexp}
(1-\e) (\lambda(r)-\e) \leq \frac{\N_R}{\Vol(B_R)} \leq (1+\e)(\lambda(r) + a(r) + \e) \quad \text{in expectation.}
\ee
Using $a(r)< \e$ and also that $\lambda(r) \leq 1$ for all $r$,
which follows from
$$\EE \N(x,r) \leq \EE \# P \cap B(0,r) = \Vol(B(0,r)),$$
\eqref{eq:inexp} implies
$$ \EE \left| \frac{\N_R}{\Vol(B_R)} - \lambda(r) \right| < \e(1+\lambda(r)) + (1+\e)a(r) < 2\e+(1+\e)\e. $$
Since $\e>0$ was arbitrary, 
this implies the existence of a constant $c_\gamma$
such that $\lambda(r) \rightarrow c_\gamma$
in $L^1$ as $r \rightarrow \infty$
and $\N_R/\Vol(B_R) \rightarrow c_\gamma$
in $L^1$ as $R \rightarrow \infty$.

The proof of the second
statement in the proposition
concerning the existence 
of a limit $c_R \rightarrow c$
for all components (with no restriction on homotopy type)
follows from the same argument
while replacing the integral geometry sandwich with its more basic version (see Remark \ref{rmk:IGS}
in Section \ref{sec:prelim}).
\end{proof}

\subsection{Proof of Theorem \ref{thm:mainpoisson}}\label{sec:pfpoisson}
We write the measure $\mu_R$ as:
\begin{align}\mu_R&=\frac{1}{\N(\Pc, B(0, R))}\sum_{\gamma\in \G}\N(\Pc, B(0, R);\gamma)\delta_\gamma\\
&=\frac{\mathrm{Vol}(B(0,R))}{\N(\Pc, B(0, R))}\sum_{\gamma\in \G}\frac{\N(\Pc, B(0, R);\gamma)}{\mathrm{Vol}(B(0, R))}\delta_\gamma\\
&=\sum_{\gamma\in \G}\frac{c_{R, \gamma}}{c_R}\delta_\gamma.
\end{align}

By the convergence statements in Proposition \ref{prop:ergodic},
and since $c>0$
(which follows from Proposition \ref{prop:alleuclid} since $c \geq c_\gamma$),
we have $a_{R,\gamma}=\frac{c_{R, \gamma}}{c_R}$ converges in $L^1$ to a constant $a_\gamma = \frac{c_\gamma}{c}$ as $R \rightarrow \infty$.

The positivity of the coefficients $a_\gamma$ follows from the positivity of $c_\gamma$ stated in Proposition \ref{prop:alleuclid}.

Next we prove that the measure
$$\mu = \sum_{\gamma \in \G} a_\gamma \delta_\gamma $$
is indeed a probability measure (see Proposition \ref{prop:probmeas} below).
The main obstacle, addressed in the following lemma (cf. \cite[Sec. 4]{SarnakWigman}), 
is in preventing 
the mass in the sequence of measures
$\mu_R$ from escaping to infinity.

\begin{lemma}[Topology does not leak to infinity]\label{lemma:notleak}
For every $\delta> 0$
there exists a finite set $g \subset \G$
and $R_0>0$ such that for all $R \geq R_0$
$$ \EE \sum_{\gamma \in g^c} c_{R,\gamma} < \frac{\delta}{4} .$$
\end{lemma}

\begin{proof}[Proof of Lemma \ref{lemma:notleak}]
First we observe that
\begin{equation}\label{eq:cvgcsmallr}
\sum_{\gamma \in \G} \frac{\EE N(\Pc,B(0,r);\gamma)}{\Vol\left(B(0,r) \right)} < a_0 < \infty,
\end{equation}
where $a_0$ is independent of $r$.
Indeed,
$$ \sum_{\gamma \in \G} \frac{\EE N(\Pc,B(0,r);\gamma)}{\Vol\left(B(0,r) \right)} = \frac{\EE N(\Pc,B(0,r))}{\Vol\left(B(0,r) \right)}
\leq \frac{\EE |\{ P \cap B(0,r) \}|}{\Vol\left(B(0,r) \right)},$$
which is a constant independent of $r$
(the average number of points of a Poisson process in a given region is proportional to the volume of the region).

Let $A \subset \G$ be arbitrary.
Then, using the Integral Geometry Sandwich, we obtain:
\begin{align}
	\int_{\Omega} \sum_{\gamma \in A} c_{R,\gamma}(\omega) d\omega &\leq
	\left( 1+ \frac{r}{R} \right)^d \frac{1}{\Vol(B_{R+r})} \EE \left\{ \sum_{\gamma \in A} \int_{B_{R+r}} \frac{N^*(\Pc,B(x,r);\gamma)}{\Vol(B_r)} dx \right\} \\
	&\leq
	\left( 1+ \frac{r}{R} \right)^d \left( \sum_{\gamma \in A} \EE \frac{N(\Pc,B(x,r);\gamma)}{\Vol(B_r)} + O(r^{-1}) \right).
\end{align}

Let $\delta > 0$ be arbitrary,
and choose $r$ sufficiently large
that the above $O(r^{-1})$
error term is smaller than
$\delta/16$.

By the convergence \eqref{eq:cvgcsmallr}
there exists a finite
set $g \subset \G$
such that
\begin{equation}\label{eq:cofinite}
\sum_{\gamma \in g^c} \frac{\EE N(\Pc,B(0,r);\gamma)}{\Vol\left(B(0,r) \right)} < \frac{\delta}{16}.
\end{equation}

Choosing $R_0$ large enough
that $\displaystyle \left( 1+ \frac{r}{R} \right)^d < 2$
we then have for all $R \geq R_0$
$$ \EE \sum_{\gamma \in g^c} c_{R,\gamma} < 2\left( \frac{\delta}{16} + \frac{\delta}{16} \right) = \frac{\delta}{4},$$
as desired, and this completes the proof of the lemma.
\end{proof}

\begin{prop}\label{prop:probmeas}
The measure 
$$\mu= \sum_{\gamma \in \G} a_\gamma \delta_{\gamma},$$
with $a_\gamma = \frac{c_\gamma}{c}$,
is a probability measure.
\end{prop}

\begin{proof}[Proof of Proposition \ref{prop:probmeas}]
We need to show that 
$\sum_{\gamma \in \G} a_\gamma = 1$,
or equivalently,
\be\label{eq:goalcgamma}
\sum_{\gamma \in \G} c_\gamma = c.
\ee

Let $\delta>0$ and take $g \subset \G$ to be the set guaranteed by Lemma \ref{lemma:notleak}.

We want to show that
\begin{equation}\label{eq:cgammagoal}
\left|\sum_{\gamma \in \G} c_\gamma - c \right| \leq \delta,
\end{equation}
which will then immediately establish 
\eqref{eq:goalcgamma} since $\delta>0$ is arbitrary.

For fixed $\omega$ in the sample space $\Omega$
observe that by Fatou's lemma 
$$\sum_{\gamma \in g^c} c_\gamma \leq \liminf_{R \rightarrow \infty} \sum_{\gamma \in g^c} c_{R,\gamma}(\omega),$$
and applying Fatou's lemma again followed by Tonelli's theorem,
we have
\begin{align}
\int_{\Omega} \sum_{\gamma \in g^c} c_\gamma d\omega &\leq \int_{\Omega} \liminf_{R \rightarrow \infty} \sum_{\gamma \in g^c} c_{R,\gamma}(\omega) d\omega \\
&\leq \liminf_{R \rightarrow \infty} \int_{\Omega}  \sum_{\gamma \in g^c} c_{R,\gamma}  d\omega \\
\quad &= \liminf_{R \rightarrow \infty} \sum_{\gamma \in g^c} \int_{\Omega} c_{R,\gamma}(\omega) d\omega.
\end{align}

Combining this with Lemma \ref{lemma:notleak} 
we obtain
\begin{equation}\label{eq:cgammatail}
\int_{\Omega} \sum_{\gamma \in g^c} c_\gamma d\omega \leq \frac{\delta}{4}.
\end{equation}

We proceed to estimate
$\displaystyle \left|\sum_{\gamma \in \G} c_\gamma - c \right|$:

\begin{align}
 \left|\sum_{\gamma \in \G} c_\gamma - c \right| &= \int_{\Omega}  \left|\sum_{\gamma \in \G} c_\gamma - c \right|  d\omega \\
&= \int_{\Omega}  \left|\sum_{\gamma \in \G} c_\gamma - c_{R} + c_{R} - c \right|  d\omega \\
&= \int_{\Omega}  \left|\sum_{\gamma \in \G} c_\gamma - \sum_{\gamma \in \G} c_{R,\gamma} + c_{R} - c \right|  d\omega \\
&\leq  \int_{\Omega} \left|\sum_{\gamma \in \G} c_\gamma - \sum_{\gamma \in \G} c_{R,\gamma} \right| d\omega + \int_{\Omega} \left|c_{R} - c  \right| d\omega\\
&\leq  \int_{\Omega} \left| \sum_{\gamma \in g} c_\gamma - \sum_{\gamma \in g} c_{R,\gamma} \right| d \omega + \int_{\Omega} \left| \sum_{\gamma \in g^c} c_\gamma - \sum_{\gamma \in g^c} c_{R,\gamma} \right| d\omega + \int_{\Omega} \left|c_{R} - c  \right| d\omega\\
&\leq  \int_{\Omega} \sum_{\gamma \in g} \left| c_\gamma - c_{R,\gamma} \right| d \omega + \int_{\Omega} \sum_{\gamma \in g^c} c_\gamma d \omega + \int_{\Omega} \sum_{\gamma \in g^c} c_{R,\gamma} d\omega + \int_{\Omega} \left|c_{R} - c  \right| d\omega\\
&\leq \frac{\delta}{4} + \frac{\delta}{4} + \frac{\delta}{4} + \frac{\delta}{4}.
\end{align}

In the last line, we have estimated the first and last terms by $\delta/4$ by choosing $R$ sufficiently large
and using the $L^1$ convergence $c_{R,\gamma} \rightarrow c_\gamma$ (and the fact that $g$ is a finite set) and the $L^1$ convergence
$c_R \rightarrow c$;
we have estimated the second term by $\delta / 4$ using \eqref{eq:cgammatail};
and we have estimated the third term by $\delta/4$ by choosing $R$ sufficiently large to apply Lemma \ref{lemma:notleak}.
This establishes \eqref{eq:cgammagoal} and concludes the proof of the proposition.
\end{proof}

Having established by Proposition \ref{prop:probmeas} that $\mu$ is a probability measure, the convergence in probability $\mu_R \rightarrow \mu$ now follows from the coefficient-wise convergence $a_{R,\gamma} \rightarrow a_\gamma$ along with the purely measure-theoretic result Lemma \ref{lemma:cvg}.

The statement that the support of $\mu$ is all of $\G$ follows from the positivity of the coefficients $a_\gamma$.
This concludes the proof of Theorem \ref{thm:mainpoisson}.

\section{Semi-local counts in the Riemannian case}\label{sec:semilocal}
In this section,
we study the components of $\U_n$ contained in  a neighborhood $\B(p, R n^{-1/d})$ of a point $p\in \M$ by relating this case to the Euclidean case.
As a preliminary step, we consider the following diffeomorphism (see also Proposition \ref{prop:stability}):
\be\label{eq:psi} \psi_n:\B(p, Rn^{-1/d})\xrightarrow{\mathrm{exp}_{p}^{-1}}B_{T_{p}M}(0, Rn^{-1/d})\xrightarrow{n^{1/d}}B_{T_{p}M}(0, R)\simeq B(0, R).\ee 
Through $\psi_n$, the stochastic point process $U_n\cap \B(p, Rn^{-1/d})$ induces a stochastic point process on $B(0, R)$ which converges in distribution to the uniform Poisson process on $B(0, R)$ \cite[Sec. 3.5]{Resnick}. By Skorokhod's representation theorem \cite[Ch. 1, Sec. 6]{Billingsley}, there exists a representation, or ``coupling'', of these point processes defined on a common probability space such that the convergence of these stochastic processes is almost sure.

\begin{thm}\label{thm:coupling}
Let $p\in \M$. For every $\delta>0$ and for $R>0$ sufficiently large there exists $n_0$ such that 
(using the coupling given by Skorokhod's theorem mentioned above) for every $\gamma \in \G$ and for $n\geq n_0$:
\be\label{eqdouble} \PP\left\{\N(\Pc, B(0,R); \gamma)=\N(\U_n, \B(p, R n^{-1/d}); \gamma)\right\}\geq 1-\delta.\ee
\end{thm}

\begin{proof} 
Let $\varphi_n$ denote the inverse of 
the map $\psi_n$ defined in \eqref{eq:psi}.
For the proof of \eqref{eqdouble} we will need to establish the following three facts:
\begin{enumerate}
\item There exists $\ell_0>0$ and $n_1>0$ such that with probability at least $1-\delta/3$ we have:
\be \label{eq:points}\#(U_n\cap \B(p, Rn^{-1/d}))=\#(P \cap B(0, R))\leq \ell_0\ee
(i.e. with positive probability for large $n$, depending on $\delta$,  both point processes have the same number of points and this number is bounded by some constant $\ell_0$, which also depends on $\delta$).
\item There exists\footnote{The symbol $\coprod A_j$ denotes the disjoint union of the sets $A_j$. }  $W\subset \coprod_{\ell\leq \ell_0}B(0, R)^\ell$ , $r>0$ and $n_2\geq n_1>0$ such that $\PP(W)\geq 1-\delta/3$ and  for every $x=(y_1, \ldots, y_\ell)\in W$ if $\tilde{x}=(\tilde{y}_1, \ldots, \tilde{y}_\ell)$ is such that $\|x-\tilde x\|< r$ and  $n\geq n_2$ then:
\be\label{eq:cp} \bigcup_{k=1}^\ell B(y_i, \alpha)\simeq \bigcup_{k=1}^\ell \B(\varphi_n(\tilde y_i), \alpha n^{-1/d}),\ee
(i.e. the two spaces are homotopy equivalent), and for every connected component of $\bigcup_{k=1}^\ell B(y_i, \alpha)$ this component intersects $\partial B(0, R)$ if and only if the corresponding component of $\bigcup_{k=1}^\ell \B(\varphi_n(\tilde y_i), \alpha n^{-1/d})$ intersects $\partial \B(p, Rn^{-1/d}).$

Let us explain this condition. A point $x=(y_1, \ldots, y_\ell)$ in $W$ corresponds to the spatial Poisson event: we sample $\ell$ points and each of these points is sampled uniformly from $B(0, R)$. With probability at least $1-\delta/3$, by the previous point, the number of samples of the spatial Poisson distribution is at most $\ell_0.$ Each such  Poisson sample in $\R^d$ gives rise to a geometric complex in $\R^d$, and this complex is nondgenerate with probability one. Given such a nondegenerate complex $\bigcup_{k=1}^\ell B(y_i, \alpha)$ in $\R^d$, we can perturb ``a little'' (how little is quantified by ``$r>0$'') the point $x=(y_1, \ldots, y_\ell)$ to a point $\tilde{x}=(\tilde{y}_1, \ldots, \tilde{y}_\ell)$ inside $W$ and still get a nondegenerate complex which has the same homotopy type of the original one. Now, to each point $\tilde{x}\in W$ there also corresponds a complex $\bigcup_{k=1}^\ell \B(\varphi_n(\tilde y_i), \alpha n^{-1/d})$ in the manifold $M$, through the map $\varphi_n$. When $n$ is ``large enough'', it is natural to expect that the two complexes $\bigcup_{k=1}^\ell B(y_i, \alpha)$ and $\bigcup_{k=1}^\ell \B(\varphi_n(\tilde y_i), \alpha n^{-1/d})$ have the same homotopy type. Point (2) says that, given $\ell_0$, we can find $W$ with probability at least $1-\delta/3$, $r>0$ small enough and $n>0$ large enough such that this is true. (We also added the requirement that the intersection with the boundary of the big containing ball is the same, but the essence of point (2) is in the requirement that the two complexes should have the same homotopy type.)
\item Assuming point (1), denoting  by $\{x_1, \ldots, x_\ell\}=P_R\cap B(0, R)$,  and by $\{\tilde{x}_1, \ldots, \tilde{x}_\ell\}=\psi_n(U_n\cap \B(p, Rn^{-1/d}))$, there exists $n_3\geq n_1>0$ such that for every $n\geq n_3$:
\be\PP\left\{
\forall \ell\leq\ell_0,\, \forall k=1, \ldots, \ell, \,\|x_k-\tilde{x}_k\|\leq r\right\}\geq 1-\delta/3.
\ee
\end{enumerate}
Assuming these three facts, \eqref{eqdouble} follows arguing as follows. With probability at least $1-\delta$ for $n\geq n_0=\max\{n_1, n_2, n_3\}$ all the conditions from (1), (2) and (3) verify and the two random sets \be\bigcup_{p\in \Pc_R} B(p, \alpha)\quad \textrm{and}\quad   \bigcup_{p_k\in U_n\cap \B(p, Rn^{-1/d})} \B(p_k, \alpha n^{-1/d})\ee
are homotopy equivalent and by the second part of point (2) also the unions of all the components entirely contained in $B(0, R)$ (respectively $\B(p, Rn^{-1/d})$) are homotopy equivalent. In particular the number of components of a given homotopy type $\gamma$ is the same for both sets with probability at least $1-\delta$.

It remains to prove (1), (2) and (3). 

Point (1) follows from the fact that 
(working with the representation provided by Skorokhod's theorem),
the point process $\psi_n(U_n\cap \B(p, Rn^{-1/d}))$ converges almost surely to the Poisson point process on $B(0, R)$. In particular the sequence of random variables $\{\#(U_n\cap \B(p, Rn^{-1/d}))\}$ converges almost surely to $\#(P_R\cap B(0, R))$ and \eqref{eq:points} follows from the fact that almost sure convergence implies convergence in probability.

For point (2) we argue as follows. Given $\ell_0$ we consider the compact semialgebraic set:
\be X=\coprod_{\ell\leq\ell_0}B(0, R)^\ell.\ee
This set is endowed with the measure $d\rho$:
\be d\rho=\sum_{\ell\leq \ell_0}\frac{\mathrm{vol}(B(0, R)^\ell)}{\ell!}\chi_{B(0, R)^\ell}d\lambda_{B(0, R)^\ell}\ee
where $d\lambda$ denotes the Lebesgue measure (this is the measure induced from the Poisson distribution).

Let now $Z\subset X$ be the set of points $x=(y_1, \ldots, y_\ell)$ such that either the intersection $\bigcap_{j\in J_1} \partial B(y_{j},\alpha)$ or the intersection $\partial B(0, R)\bigcap_{j\in J_2} \partial B(y_{j},\alpha)$ is non-transversal for some index sets $J_1, J_2\in \genfrac{\{ }{\}}{0pt}{}{d}{\ell}$ (note that the generic intersection of  more than $d$ spheres will be empty). This set $Z$ is also a semialgebraic set, and it has measure zero: it cannot contain any open set, because the nondegeneracy condition is open and dense.

Let $U(Z)$ be an open neighborhood of $Z$ such that $\rho(U(Z)^c)\geq 1-\delta/3$ (for example one can take $U(Z)=\coprod_{\ell\leq \ell_0}\{d(\cdot, Z)< \e\}$ for $\e>0$ small enough). 
We set $W=U(Z)^c$ (note that $\PP(W)\geq 1-\delta/3$).

We will first argue that for every $x\in U(Z)^c\subset W$ we can find $r_2(x)>0$ and $n(x)>0$ such that the two complexes  \eqref{eq:cp} have the same homotopy type (and the same combinatorics of intersection with the boundary of the big containing ball) whenever $\|\tilde x-x\|<r_2(x)$ and $n>n(x).$ Then we will use the compactness of $U(Z)^c$ in order to find uniform $r>0$ and $n>0.$

Pick therefore $x=(y_1, \ldots, y_\ell)\in U(Z)^c$. The property of transversal intersection implies that for every index set $J_1\in \bigcup_{\ell\leq\ell_0}\genfrac{\{ }{\}}{0pt}{}{d}{\ell}$ such that the intersection $\cap_{j\in J_1}B(y_j, \alpha)$ is nonempty, this intersection contains a nonempty open set, and there exists a point $\sigma_{J_1}(x)$ such that for every $j\in J_1$ we have $\|y_j-\sigma_{J_1}(x)\|<\alpha.$ Similarly for every $ J_2\in \bigcup_{\ell\leq\ell_0}\genfrac{\{ }{\}}{0pt}{}{d}{\ell}$ whenever an intersection $\partial B(0, R)\bigcap_{j\in J_2} \partial B(y_{j},\alpha)$ is transversal and nonempty, there exists a point $\sigma_{J_2}(x)$ such that $\|\sigma_{J_2}(x)\|>R$ and for every $j\in J_2$ we have $\|y_j-\sigma_{J_2}(x)\|<\alpha.$ 
Because these are open properties, there exists $r_1(x), r_2(x)>0$ such that for every $w=(w_1, \ldots, w_\ell)$ and $z=(z_1, \ldots, z_\ell)$ with $\|w_j-y_j\|\leq r_1(x)$ and $\|z_j-w_j\|<r_2(x)$ for all $j=1, \ldots, \ell$, we have:
\be \forall j\in J_1: \,\|z_j-\sigma_{J_1}(x)\|<\alpha\quad \textrm{and}\quad \forall j\in J_2:\,\|z_j-\sigma_{J_2}(x)\|<\alpha.\ee
Moreover since the property of having non-empty intersection is also stable under small perturbations, we can assume that $r_1(x), r_2(x)$ are small enough to guarantee also that: 
\be \bigcap_{j\in J_3}B(z_j, \alpha)=\emptyset \iff \bigcap_{j\in J_3}B(x_j, \alpha)=\emptyset.\ee

Observe now that the sequence of functions $d_n:B(0, R)\times B(0, R)\to \R$ defined by:
\be d_n(x_1, x_2)=d_{\M}(\varphi_n(x_1), \varphi_n(x_2))n^{1/d}\ee
converges uniformly to the Euclidean distance in $\R^d$. In particular there exists $n(x)>0$ such that for every $n\geq n(x)$, for every $w=(w_1, \ldots, w_\ell)$ and $z=(z_1, \ldots, z_\ell)$ with $\|w_j-y_j\|\leq r_1(x)$ and $\|z_j-w_j\|<r_2(x)$, for all $j=1, \ldots, \ell$  and for $i=1, 2$ we have:
\be d_{\M}(\varphi_n(z_j), \varphi_n(\sigma_{J_i}(x)))<\alpha n^{-1/d}.\ee
Moreover, for a possibly larger $n(x)$, we also have that 
\be \bigcap_{j\in J_3}\B(\varphi_n(z_j), \alpha n^{-1/d})=\emptyset \iff \bigcap_{j\in J_3}B(x_j, \alpha)=\emptyset.\ee
Choosing $n(x)$ to be even larger, so that balls of radius smaller than $\alpha n^{-1/d}$ in $\M$ are geodesically convex, these conditions imply that the combinatorics of the covers
\be \{B(x_j, \alpha)\}_{j=1}^\ell\quad \textrm{and}\quad \{\B(\varphi_n(z_j), \alpha n^{-1/d})\}_{j=1}^\ell\ee
are the same and, by Lemma \ref{lemma:convex}, the two sets $\bigcup_{j\leq \ell}B(x_j, \alpha)$ and  $\bigcup_{j\leq \ell}\B(\varphi_n(z_j), \alpha n^{-1/d})$ are homotopy equivalent. Also, the above condition on $\sigma_{J_2}(x)$ implies that a component of $\bigcup_{j\leq \ell}B(x_j, \alpha)$ intersects $\partial B(0, R)$ if and only if the corresponding component of $\bigcup_{j\leq \ell}\B(\varphi_n(z_j), \alpha n^{-1/d})$ intersects $\B(p, Rn^{-1/d}).$

Finally, we cover now $W=X\backslash U(Z)$ with the family of open sets $\bigcup_{x\in W}B(x, r_1(x))$ and find, by compactness of $W$, finitely many points $x_1, \ldots, x_L$ such that the union of the balls  $B(x_k, r_1(x_k))$ with $k=1, \ldots, L$ covers $W$.  With the choice $n_2=\max\{n(x_k), k=1, \ldots, L\}$ and $r=\min\{r_2(x_k), k=1, \ldots, L\}$ property (2) is true.

Concerning point (3), we observe that again this follows from the fact that the point process $\psi_n(U_n\cap \B(p, Rn^{-1/d}))$ converges almost surely (hence in probability) to the Poisson point process on $B(0, R)$.
\end{proof}

\begin{cor}\label{cor:doublescaling}
For each $\gamma \in \G$, $\alpha>0$, $x \in \M$, 
and $\e>0$, we have
$$ \lim_{R \rightarrow \infty} \limsup_{n \rightarrow \infty} \PP \left\{ \left| \frac{\N(\U_n, \hat{B}(x,R n^{-1/d});\gamma)}{\Vol(B_R)} - c_\gamma \right| > \e \right\} = 0 .$$
\end{cor}
\begin{proof}
This follows from Theorem \ref{thm:coupling}
combined with Proposition \ref{prop:ergodic}.
Indeed, let $\e>0$ and $\delta>0$ be arbitrary.
By Proposition \ref{prop:ergodic}
there exists $R_0$ such that for $R > R_0$
we have
\be
\PP\left\{ \left|\frac{\N(\Pc, B(x,R);\gamma)}{\Vol(B_R)} - c_\gamma \right| > \e \right\} < \delta.
\ee
Fix any such $R>R_0$.
The event 
\be 
\left| \frac{\N(\U_n, \hat{B}(x,R n^{-1/d});\gamma)}{\Vol(B_R)} - c_\gamma \right|>\e
\ee
is contained in the union of the event $E_1$ that
\be
\left|\frac{\N(\Pc, B(x,R);\gamma)}{\Vol(B_R)} - c_\gamma \right| > \e
\ee
and another event $E_\delta$,
which is the event that
$\N(\Pc, B(x,R);\gamma) \neq \N(\U_n, \hat{B}(x,R n^{-1/d});\gamma) $.
Thus,
\be\label{eq:union}
\PP \left\{ \left| \frac{\N(\U_n, \hat{B}(x,R n^{-1/d});\gamma)}{\Vol(B_R)} - c_\gamma \right| > \e \right\} \leq \PP \{ E_1 \} + \PP \{ E_\delta \} < \delta + \PP \{ E_\delta \}.
\ee

By Theorem \ref{thm:coupling}, there exists $n_0$ such that
for all $n \geq n_0$ we have $\PP \{ E_\delta \} \leq \delta$.
 
Thus, applying this to \eqref{eq:union} we obtain
\be
\limsup_{n \rightarrow \infty} \PP \left\{ \left| \frac{\N(\U_n, \hat{B}(x,R n^{-1/d});\gamma)}{\Vol(B_R)} - c_\gamma \right| > \e \right\} < 2\delta.
\ee
Since $\delta>0$ was arbitrary,
this completes the proof of Corollary \ref{cor:doublescaling}.
\end{proof}

\section{The global count for the Riemmanian case: 
proof of Theorem \ref{thm:main1}}\label{sec:global}

In this section we establish the limit law in the manifold setting (cf. \cite[Sec. 7]{SarnakWigman}).
As in the Euclidean case,
the main step is to prove coefficient-wise convergence,
which is stated in the following theorem.

\begin{thm}\label{thm:gammafixed}
For every $\gamma \in \G$, the random variable
\be c_{n, \gamma}=\frac{\N(\U_n,\M;\gamma)}{n}\ee
converges in $L^1$ to the constant $c_\gamma = c_\gamma(\alpha)$
(the same constant as in Proposition \ref{prop:ergodic}). The same statement is true for the random variable \be c_n=\frac{\mathcal{N}(\mathcal{U}_n, M)}{n}\ee (i.e. when we consider all components, with no restriction on their type): as $n\to \infty$, it converges in $L^1$ to the constant $c=\sum_{\gamma\in \G}c_\gamma$. 
\end{thm}

\subsection{Proof of Theorem \ref{thm:main1} assuming Theorem \ref{thm:gammafixed}}
Since convergence in $L^1$ implies convergence in probability, Theorem \ref{thm:gammafixed} ensures that the random variable $c_{n,\gamma}=\frac{\mathcal{N}(\mathcal{U}_n, M;\gamma)}{n}$ converges in probability to the constant $c_\gamma$; similarly the random variable $c_n=\frac{\mathcal{N}(\mathcal{U}_n, M)}{n}$ converges in $L^1$ (hence in probability) to $c>0$.
The proof now proceeds similarly to the proof of Theorem \ref{thm:mainpoisson}. We write the measure $\hat\mu_n$ as:
\begin{align}
\hat\mu_n&=\frac{1}{b_0(\Check{C}(\U_n))}\sum_{\gamma\in \hat{G}}\N(\U_n, M;\gamma)\delta_\gamma\\
&=\frac{1}{b_0(\Check{C}(\U_n))}\left(\sum_{\gamma\in \G}\N(\U_n, M;\gamma)\delta_\gamma+\sum_{\gamma\in \hat\G\backslash \G}\N(\U_n, M;\gamma)\delta_\gamma\right)\\
&=\frac{1}{\N(\U_n, M)}\sum_{\gamma\in \G}\N(\U_n, M;\gamma)\delta_\gamma+\frac{1}{\N(\U_n, M)}\sum_{\gamma\in \hat\G\backslash \G}\N(\U_n, M;\gamma)\delta_\gamma\\
&=\frac{n}{\N(\U_n, M)}\sum_{\gamma\in \G}\frac{\N(\U_n, M;\gamma)}{n}\delta_\gamma+\frac{n}{\N(\U_n, M)}\sum_{\gamma\in \hat\G\backslash \G}\frac{\N(\U_n, M;\gamma)}{n}\delta_\gamma\\
\label{eq:laststep}&=\sum_{\gamma\in \G}\frac{c_{n, \gamma}}{c_n}\delta_\gamma+\sum_{\gamma\in \hat\G\backslash\G}\frac{c_{n, \gamma}}{c_n}\delta_\gamma.
\end{align}

We have $a_{n,\gamma}=\frac{c_{n, \gamma}}{c_n}$ converges in $L^1$ to the constant
$a_\gamma=\frac{c_\gamma}{c}$.
Recalling that the measure
$$\mu = \sum_{\gamma \in \G} a_\gamma \delta_\gamma $$
is a probability measure (see Proposition \ref{prop:probmeas}),
we can apply Lemma \ref{lemma:cvg}
to conclude that the measure on the left in \eqref{eq:laststep} converges in probability to $\mu$.

Since $\mu$ is a probability measure,
this implies that
$$\sum_{\gamma\in \G}\frac{c_{n, \gamma}}{c_n}$$
converges to $1$ in probability.
For any $\gamma_0 \in \hat\G\backslash\G$
this implies that the coefficient $\frac{c_{n, \gamma_0}}{c_n}$ appearing in the measure 
on the right in \eqref{eq:laststep} converges to zero in probability,
since 
$$0 \leq \frac{c_{n, \gamma_0}}{c_n} \leq 1 - \sum_{\gamma\in \G}\frac{c_{n, \gamma}}{c_n}.$$
Thus, the measure $\hat\mu_n$ converges in probability to $\mu$ by another application of Lemma \ref{lemma:cvg}.

\subsection{Proof of Theorem \ref{thm:gammafixed}}
{\bf Note:} Since $\alpha>0$ and $\gamma \in \G$ are fixed, we will simply use
\begin{equation}\label{eq:simplenotation}
 \N_n := \N(\U_n,\M;\gamma)
\end{equation}
to denote the number of components of $\U_n$ in $\M$ of type $\gamma$.
We will use
$$\N^*_n(x,r) := \N^*(\U_n,\hB(x,r);\gamma)$$
to denote the number of such components intersecting the geodesic ball $\hB(x,r)$ of radius $r$ centered at $x$
and 
$$\N_n(x,r):=\N(\U_n,\hB(x,r);\gamma)$$
to denote the number of components 
completely contained in $\hB(x,r)$.

Thus, our goal, stated in the abbreviated notation \eqref{eq:simplenotation},
is to prove
\be\label{eq:goalmanif}
\EE \left[ \left| \frac{\N_n}{n} - c_\gamma \right| \right] \rightarrow 0.
\ee

Using the integral geometry sandwich 
from Theorem \ref{thm:IGSmanif} 
we have
\be \label{eq:sandn}
(1-\e) \int_{\M} \frac{\N_n(x,Rn^{-1/d})}{\Vol \left( B_{R} \right)} dx
\leq \frac{\N_n}{n} \leq (1+\e) \int_{\M} \frac{\N^*_n(x,R n^{-1/d})}{\Vol \left( B_{R} \right)} dx.
\ee 
Letting $I_1$ denote the integral on the left side and $I_2$ the one on the right side, we subtract $I_1$ from each part of \eqref{eq:sandn} and write
\be \label{eq:modsandwich}
-\e I_1
\leq \frac{\N_n}{n} - I_1 \leq
\e I_1 + (1+\e) (I_2 - I_1).
\ee
In order to estimate $I_2 - I_1$
we note that the number of
connected components of $\U_n$ that intersect,
but are not completely contained in, 
the geodesic ball $\hB(x,R n^{-1/d})$
is bounded above by the number of points that fall within distance $\alpha n^{-1/d}$ to the boundary $\p \hB(x,R n^{-1/d})$.
This $\alpha n^{-1/d}$-neighborhood of $\p \hB(x,R n^{-1/d})$ is the same as the geodesic annulus centered
at $x$ with inner radius $(R-\alpha)n^{-1/d}$
and outer radius $(R+\alpha)n^{-1/d}$.
The average number of points in this annulus
equals its volume which can be estimated
(uniformly over $x \in \M$)
by that of the Euclidean annulus,
and this gives
\begin{equation}\label{eq:I2minusI1}
\EE |I_2 - I_1| = O(R^{-1}).
\end{equation}
which together with \eqref{eq:modsandwich} implies
\begin{equation}\label{eq:addedanewone1}
 \EE \left| \frac{\N_n}{n} - I_1 \right|
= O(\e) + O(R^{-1}),
\end{equation}
where we have also
used $I_1 \leq \frac{1}{1-\e}$
which follows from the first inequality in \eqref{eq:sandn} along with the simple estimate $\N_n \leq n$.

By \eqref{eq:addedanewone1} we obtain
\begin{align}
\EE \left[ \left| \frac{\N_n}{n} - c_\gamma \right| \right] &= \EE \left[ \left| \frac{\N_n}{n} - I_1 + I_1 - c_\gamma \right| \right] \\
&\leq \EE \left[ \left| I_1 - c_\gamma \right| \right] + O(\e) + O(R^{-1}). \\
&= \EE \left[ \left| \int_\M \frac{\N_n(x,R n^{-1/d})}{\Vol(B_{R})} - c_\gamma \, dx \right| \right] + O(\e) + O(R^{-1}).
\end{align}

Thus, in order to prove the theorem it suffices to show that
the above term $\EE \left[ \left| I_1 - c_\gamma \right| \right] $
can be made arbitrarily small for all sufficiently large
$n$.

Define the ``bad'' event
$$ \Omega_{x,R,n} := \left\{ \left| \frac{\N_n(x,R n^{-1/d})}{\Vol (B_{R})} - c_\gamma \right| > \e \right\}.$$

Claim: There exists a sequence
$R_j \rightarrow \infty$
such that for every $\delta>0$
there exists $\M_\delta \subset \M$
with $\Vol (\M_\delta) > 1 - \delta$
such that
\be\label{eq:EgorovClaim}
\lim_{R_j \rightarrow \infty} \limsup_{n \rightarrow \infty}
\sup_{x \in \M_\delta} \PP \left( \Omega_{x,R_j,n} \right) = 0.
\ee
 
The proof of this claim closely follows 
\cite{SarnakWigman} and uses Egorov's 
theorem as well as the idea from the proof of Egorov's theorem.
We start by recalling the point-wise 
limit stated in Corollary \ref{cor:doublescaling}.
For each $x \in \M$, we have
$$ \lim_{R \rightarrow \infty} \limsup_{n \rightarrow \infty} \PP \left\{ \Omega_{x,R,n} \right\} = 0.$$
Let us restrict to $R \in \NN$.
Apply Egorov's theorem
to obtain $M_{\delta}' \subset M$
with $\Vol(M_{\delta}') > 1- \frac{\delta}{2}$ such that
\begin{equation}\label{eq:Egorov1}
\lim_{R \rightarrow \infty} \sup_{x \in M_{\delta}'} \limsup_{n \rightarrow \infty} \PP \left\{ \Omega_{x,R,n} \right\} = 0.
\end{equation}

Next we use an additional Egorov-type argument
in order to obtain the statement
in the claim (where we will obtain the set
$M_\delta$ by slightly shrinking $M_\delta'$).
For each fixed integer $j>0$,
we can find by \eqref{eq:Egorov1} an
$R_j \in \NN$ sufficiently large so that
\begin{equation}\label{eq:Egorov2}
\sup_{x \in M_{\delta}'} \limsup_{n \rightarrow \infty} \PP \left\{ \Omega_{x,R_j,n} \right\} < \frac{1}{j}.
\end{equation}

Letting $F_{m}(j)$ denote the monotone decreasing (with $m$) sequence of sets
$$F_{m}(j) = \bigcup_{k \geq m} \left\{ x \in M_\delta' : \PP (\Omega_{x,R_j,k}) > \frac{2}{j} \right\},$$
we see from \eqref{eq:Egorov2} that
$$\bigcap_{m \geq 1} F_m(j) = \emptyset.$$
Thus,
there exists $m=m(j)$ such that
$\Vol(F_m(j)) < \frac{\delta}{2^{j+1}}$.
We take 
$$M_\delta = M_\delta' \setminus \left( \bigcup_{j\geq 1} F_{m(j)}(j) \right), $$
which satisfies 
$\Vol(M_\delta) > \Vol(M_\delta') - \frac{\delta}{2} > 1 - \delta$.
It follows from the definition of 
$F_m(j)$ that
\begin{equation}\label{eq:Egorov3}
\limsup_{n \rightarrow \infty} 
\sup_{x \in M_{\delta}} \PP \left\{ \Omega_{x,R_j,n} \right\} \leq \frac{2}{j},
\end{equation}
and we see that \eqref{eq:EgorovClaim} 
is satisfied.

Denoting the whole probability space
as $\Omega$, we separate the integration (defining the expectation)
over the two sets $\Omegabad$ and $\Omega \setminus \Omegabad$.

\be\label{eq:A}
\EE \left[ \left| I_1 - c_\gamma \right| \right] = \int_{\Omega \setminus \Omegabad}\left|I_1- c_\gamma \right| d\omega + \int_{\Omegabad} \left|I_1- c_\gamma \right| d\omega.
\ee

We use the definition of
$\Omegabad$ to estimate the first integral in \eqref{eq:A}:

\be\label{eq:B}
\int_{\Omega \setminus \Omegabad}\left|I_1 - c_\gamma \right| d \omega 
\leq 
\int_{\Omega \setminus \Omegabad}  \int_\M \left| \frac{\N_n(x,R n^{-1/d})}{\Vol(B_{R})} - c_\gamma \right| dx d\omega \leq  \e.
\ee

For the second integral in \eqref{eq:A},
we use the estimate
\be\label{eq:minvol}
\frac{\N_n(x,R n^{-1/d})}{\Vol(B_R)} \leq (1+\e) n \xi^{-1} = O(1),
\ee
where $\xi>0$ is the minimum (over $x \in \M$) volume of a geodesic ball of radius $\alpha n^{-1/d}$,
which is uniformly (over $x \in \M$)
comparable to the volume of the Euclidean ball
of the same radius,
and hence $\xi$ is bounded below by a constant times $n^{-1}$. 
The estimate 
\eqref{eq:minvol} is based on the fact that each component has volume trivially at least
a constant times
$\alpha^d/n$,
and the fact that the minimal volume of a component times the number of components cannot exceed the volume of the region where they are contained (while fixing attention on components of type $\gamma$ as we are throughout the proof).
Applying \eqref{eq:minvol}, we obtain

\begin{align}\label{eq:C}
\int_{\Omegabad}\left|I_1 - c_\gamma \right| d \omega 
&\leq \int_\M 
\int_{\Omegabad}  \left| \frac{\N_n(x,R n^{-1/d})}{\Vol(B_{R})} - c_\gamma \right| d\omega dx \\
&\leq \left(O(1) + c_\gamma \right) \cdot \int_{\M} \PP(\Omegabad ) dx \\
&= O(1) \cdot \int_{\M} \PP(\Omegabad ) dx.
\end{align}

Next, we split this last integration over
$\M_\delta$ and $\M \setminus \M_\delta$:

\begin{align}\label{eq:D}
\int_{\M} \PP(\Omegabad) dx
&= \int_{\M_\delta} \PP(\Omegabad) dx
+  \int_{\M \setminus \M_\delta} \PP(\Omegabad) dx \\
&\leq \sup_{x \in \M_\delta} \PP(\Omegabad)
+ \delta.
\end{align}
Bringing the estimates \eqref{eq:A}, \eqref{eq:B}, \eqref{eq:C}, \eqref{eq:D} together, we have
\be
\EE \left[ \left| I_1 - c_\gamma \right| \right] \leq
\e + O(1) \left(\delta + \sup_{x \in \M_\delta} \PP(\Omegabad)\right),
\ee
which can be made arbitrarily small
using \eqref{eq:EgorovClaim}.
This establishes \eqref{eq:goalmanif}
and completes the proof of the first part of Theorem \ref{thm:gammafixed}.
The proof of the second part
concerning the count for all components (without restriction on homotopy type) follows from the same proof while replacing the integral geometry sandwich with its more basic version (see Remark \ref{rmk:IGS} in Section \ref{sec:prelim}).



\section{Positivity of all coefficients}\label{sec:quant}

In this section, we prove Propositions \ref{prop:alleuclid} and \ref{prop:all}.

\subsection{Proof of Proposition \ref{prop:alleuclid}}

Recall that, by Proposition \ref{prop:ergodic}, for every $\gamma \in \G$ we have:
\be c_\gamma=\lim_{R\to \infty}\mathbb{E}\left(\frac{\N(\Pc, B(0,R);\gamma)}{\Vol (B(0,R))}\right).\ee

The desired lower bound will come from
adding up certain local contributions
provided by the following lemma
\begin{lemma}\label{lemma:all}
Let $\Pc_0 \subset \R^d$ be a finite geometric complex. Fix $\alpha>0$. There exist $r, a>0$ (depending on $\Pc_0$ and $\alpha$) such that for any $p \in \R^d$
\be \PP\left\{\Pc \cap B(p, r)\simeq \Pc_0\right\}>a.\ee
\end{lemma}

\begin{proof}[Proof of Lemma \ref{lemma:all}]
By the translation invariance of the Poisson point process, it is enough to prove the statement while taking $p=0$.
We can assume, by Remark \ref{rmk:nondeg} above,
that $\Pc_0$ is nondegenerate, and we write
\be \Pc_0 =  \bigcup_{k=1}^\ell B(y_k, \alpha) ,\ee
where $y_1, \ldots, y_\ell\in \R^d$,
and we have taken the radius to be $\alpha$,
since we can dilate the entire set $\Pc_0$ if necessary (which does not change the homotopy type). Choose $r$ such that $\Pc_0 \subset B(0,r-\alpha)$.
By nondegeneracy, there exists $\e>0$ such that if $\|\tilde{y}_k-y_k\|\leq \e$ then the two complexes $\bigcup_{k=1}^{\ell}B(y_k, \alpha)$ and $\bigcup_{k=1}^\ell B(\tilde y_k,\alpha)$ are homotopy equivalent. 

We are thus interested in the event $E$ that for $j=1,2,...,\ell$ each ball
$B(y_j,\e)$ contains exactly one point from the random set $P=\{p_1,p_2,...\}$ and that these $\ell$ points are the only points of $P$ in the ball $B(0,r+\alpha)$.
If $E$ occurs then we have 
$\Pc \cap B(0,r) \cong \Pc_0$
as desired.
The positivity of the probability of the event $E$
can be seen by noticing that this probability is a product of
finitely many positive probabilities.
Indeed, $E$  is an intersection of the events
that $B(y_j,\e)$ contains a single point from $P$ for $j = 1,2,...,\ell$ along with the event that there are no points in the set $B(0,r+\alpha) \setminus \cup_{i=1}^\ell B(y_j,\e).$
These events are independent by a basic property of
the Poisson point process,
and each of them has positive probability
(the probabilities can be specified explicitly in terms of the volumes of the sets involved).
\end{proof}

Let $r, a$ be given by Lemma \ref{lemma:all} for the choice of $[\Pc_0]=\gamma$. 
For an appropriate $\beta>0$
we can fit $k \geq \beta \Vol(B(0,R))$
many disjoint Riemannian balls \be B_1=B(p_1, r),\, \ldots,\, B_k= B(p_k, r)\ee
in $B(0,R)$.
Observing that
\be \N(\Pc, B(0,R);\gamma)\geq\sum_{j=1}^k\N(\Pc, B_j;\gamma),\ee
we have (also using linearity of expectation)
\be \mathbb{E}\left(\frac{\N(\Pc, B(0,R);\gamma)}{\Vol (B(0,R))}\right)\geq \beta a>0,\ee
and the positivity of $c_\gamma$ follows.

\subsection{Proof of Proposition \ref{prop:all}}

The positivity of coefficients $a_\gamma$ in Theorem \ref{thm:main1}
is already established in Proposition \ref{prop:alleuclid} proved above.
In this section, we prove
the related Proposition \ref{prop:all} which provides a more direct analysis in the manifold setting.

\begin{proof}
Let $y_1, \ldots, y_\ell\in \R^d$ and $r>0$ such that 
\be \bigcup_{k=1}^\ell B(y_k, r) =\Pc\ee with $\bigcup_{k=1}^\ell B(y_k, r) $ a nondegenerate complex (it is not restrictive to consider nondegenerate complexes by Remark \ref{rmk:nondeg} above).
Let now $R'>0$ such that $B(0, R')$ contains $\bigcup_{k=1}^\ell B(y_k, r)$ and set $R=\frac{R'\alpha}{r}.$ Consider also the sequence of maps:
\be \psi_n:\B(p, Rn^{-1/d})\xrightarrow{\mathrm{exp}_p^{-1}}B_{T_pM}(0, Rn^{-1/d})\xrightarrow{\frac{r}{\alpha}n^{1/d}}B_{T_pM}(0, R')\simeq B(0, R'). \ee
Proposition \ref{prop:stability} implies that there exists $\epsilon_0>0$ and $n_0$ such that if $\|\tilde{y}_k-y_k\|\leq \epsilon_0$ then for $n\geq n_0$ the two complexes $\bigcup_{k=1}^{\ell}B(y_k, r)$ and $\bigcup_{k=1}^\ell \B(\varphi_n(\tilde y_k), \alpha n^{-1/d})$ are homotopy equivalent. 

We are interested in the the event:
\be E_{n}=\left\{\exists I_\ell\in\genfrac{\{ }{\}}{0pt}{}{n}{\ell}\,:\, \forall j\in I_\ell \quad p_{j}\in \psi_n^{-1}(B(y_j, \epsilon)), \text{ and} \quad \forall j\notin I_\ell \quad p_j\in \B(p, (R+\alpha)n^{-1/d})^c\right\}.\ee

Observe that if $E_n$ verifies, then $\U_n\cap \B(p, Rn^{-1/d})\simeq\Pc$: in fact, since there is no other point in $\B(p, (R+\alpha)n^{-1/d})$ other than $\{p_j\}_j\in I_{\ell}$, then the complex $\U_n$ is the disjoint union of the two complexes $\U_n\cap \B(p, R n^{-1/d})$
 and $\U_n\cap\B(p, (R+\alpha)n^{-1/d})^c$; the complex $\U_n\cap \B(p, R n^{-1/d})\simeq \Pc$ by Proposition \ref{prop:stability}. 

It is therefore enough to estimate from below the probability of $E_n$. Note that for every measurable subset $B\subset B(0, R')$ there exists a constant $c_B>0$ such that $\Vol\left(\psi_n^{-1}(B)\right)\geq \frac{c_B}{n}$. In particular, using the independence of the points in $U_n$, we can estimate
\begin{align}\PP(E_n)&=\binom{n}{\ell}\PP\left\{ \forall j\leq \ell: \,p_j\in \psi_n^{-1}(B(p_j, \epsilon))\quad\mathrm{and}\quad \forall j\geq \ell+1:\, p_j\in \B(p, (R+\alpha)n^{-1/d})^c \right\}\\
&=\binom{n}{\ell}\left(\prod_{j=1}^\ell\mathrm{vol}\left(\psi_n^{-1}(B(y_j, \epsilon))\right)\right) \left(\mathrm{vol}(\B(p, (R+\alpha)n^{-1/d})^c)\right)^{n-\ell}\\
&\geq \binom{n}{\ell}\left(\frac{c_1}{n}\right)^{\ell}\left(1-\frac{c_2}{n}\right)^{n-\ell}\xrightarrow{n\to \infty}\frac{c_1^\ell}{\ell !}(1-c_2)^{-\ell}e^{-c_2}.
\end{align}
In particular there exists $c>0$ such that:
\be \PP\left\{\U_n\cap \B(p, Rn^{-1/d}))\simeq \Pc\right\} \geq \PP(E_n)> c,\ee
and this concludes the proof.
\end{proof}

\bibliographystyle{abbrv}
\bibliography{Cech}

\vspace{0.1in}

{\em

Antonio Auffinger

Northwestern University, Department of Mathematics

Evanston, IL 60208}

\vspace{0.1in}

{\em

Antonio Lerario

SISSA, Mathematics

Trieste, Italy}

\vspace{0.1in}

{\em

Erik Lundberg

Florida Atlantic University, Department of Mathematical Sciences

Boca Raton, FL 33431

email: elundber@fau.edu
}

\end{document}